\documentclass[11pt]{amsart}
\usepackage{amssymb,amsmath,tabularx}
\usepackage{amsthm,verbatim}
\usepackage[bookmarks=true]{hyperref}
\usepackage{geometry}
\newtheorem{thm}{Theorem}[section]
\newtheorem{lem}[thm]{Lemma}
\newtheorem{prop}[thm]{Proposition}
\newtheorem{cor}[thm]{Corollary}
\theoremstyle{definition}

\theoremstyle{remark}
\newtheorem*{rem}{Remark}
\newtheorem*{rems}{Remarks}

\newcommand{\Lie}{\operatorname{\mathrm{Lie}}}

\newcommand{\sym}[1]{\mathfrak{S}_{#1}}
\newcommand{\Ind}{\operatorname{\mathrm{Ind}}}
\newcommand{\Res}{\operatorname{\mathrm{Res}}}
\newcommand{\F}{\mathbb{F}}
\newcommand{\Z}{\mathbb{Z}}

\begin{document}
\title{Periodic Lie modules}

\author{Kay Jin Lim}
\address[K. J. Lim]{Division of Mathematical Sciences, Nanyang Technological University, SPMS-MAS-03-01, 21 Nanyang Link, Singapore 637371.}
\email{limkj@ntu.edu.sg}
	
\author{Kai Meng Tan}
\address[K. M. Tan]{Department of Mathematics, National University of Singapore, Block S17, 10 Lower Kent Ridge Road, Singapore 119076.}
\email{tankm@nus.edu.sg}

\begin{abstract}
Let $p$ be a prime number and $k$ be a positive integer not divisible by $p$. We describe the Heller translates of the periodic Lie module $\Lie(pk)$ in characteristic $p$ and show that it has period $2p-2$ when $p$ is odd and $1$ when $p=2$.
\end{abstract}

\thanks{We thank Karin Erdmann for providing the arguments for Proposition \ref{P:just}. We are supported by Singapore Ministry of Education AcRF Tier 1 grants RG13/14 and 
R-146-000-172-112 respectively.}

\maketitle

\section{Introduction}

The main difficulty in the study of modular representation theory of finite groups is that the group algebras are not only non-semisimple but also are `almost surely' of wild representation type.  Various contraptions, such as vertices and complexities, have been invented as a measure of how non-projective the modules are and to group the modules into classes.  One of the easiest and yet non-trivial classes to study are those having complexity one, which are precisely the periodic modules.  Informally, a module is periodic if and only if the terms in its minimal projective resolution run over a repeating cycle, and the period of such a module is the length of the cycle.

The current article deals with the determination of the periods and various Heller translates of some Lie modules of the symmetric groups.  The Lie modules are classical objects of study and has appeared in \cite{We} and \cite{Wi}.  They have deep relations to algebraic topology too, as homology groups in spaces related to the braid groups \cite{Co} as well as in the study of Goodwillie calculus \cite{AM}.  More recently, the Lie modules have been tied to the finest coalgebra decomposition of the tensor algebras \cite{SW}.

Throughout this article, $\mathbb{F}$ denotes an algebraically closed field of prime characteristic $p$.

Let $n \in \mathbb{Z}^+$.  The Lie module $\Lie(n)$ of the symmetric group $\sym{n}$ of $n$ letters may be defined as the left ideal of $\mathbb{F} \sym{n}$ generated by the Dynkin-Specht-Wever element
$$
w_n = (1-d_2)(1-d_3)\dotsm(1-d_n),
$$
where $d_i$ is the descending $i$-cycle $(i,i-1,\dotsc,1)$ and we compose the elements of $\sym{n}$ from right to left.

Let $k \in \mathbb{Z}^+$ with $p \nmid k$.  In \cite[Theorem 3.2]{ELT}, the complexity of $\Lie(p^mk)$ is shown to be bounded above by $m$.  Since $\Lie(pk)$ is non-projective, this in particular shows that $\Lie(pk)$ has complexity $1$ and hence is periodic.  It is therefore natural to ask what its various Heller translates are and what its period is, and we answer these questions in this article.  More precisely, we show the following:

\begin{thm} \label{T:main}
Let $k \in \mathbb{Z}^+$ with $p \nmid k$, and let $\F_{\sym{p}}$ denote the trivial $\F \sym{p}$-module.
\begin{enumerate}
\item[(i)] For any $i \in \mathbb{Z}$, we have
$$\Omega^i (\Lie(pk)) \cong \Omega^{i+1} \left(\Ind_{\sym{p} \times \sym{k}}^{\sym{pk}} \left( \F_{\sym{p}} \boxtimes \Lie(k) \right) \right) \cong \Omega^0\left(\Ind_{\sym{p} \times \sym{k}}^{\sym{pk}} \left (\Omega^{i+1} \left(\F_{\sym{p}}\right) \boxtimes \Lie(k)\right ) \right).$$
\item[(ii)] The period of $\Lie(pk)$ is equal to the period of $\F_{\sym{p}}$, i.e.\
$$
\begin{cases}
2p-2, &\text{if $p$ is odd};\\
1, &\text{if } p=2.
\end{cases}
$$
\end{enumerate}
\end{thm}

Note that in part (i) of Theorem \ref{T:main}, the group $\sym{p} \times \sym{k}$ is identified with the subgroup $\Delta(\sym{p}) \sym{k}^{[p]}$ of $\sym{pk}$ (see subsection 2.2 for the definition of these notations).  In fact, we give a more detailed description of the Heller translates of $\Lie(pk)$ in Theorem \ref{T: Heller}.

Our approach builds on the recent work on $\Lie(pk)$ in \cite{ES}. Following \cite[Theorem 10]{ES}, we analyse the $p$th symmetrization of $\Lie(k)$, denoted as $S^p(\Lie(k))$. More precisely, we compute the various Heller translates and period of $S^p(\Lie(k))$ and these directly imply our desired results about the Lie module $\Lie(pk)$.  This in particular provides an alternative proof that the complexity of $\Lie(pk)$ is $1$, independent of the results in \cite{ELT}.  It is easier to deal with $S^p(\Lie(k))$, instead of $\Lie(pk)$, because the former is isomorphic to $\Ind_{\sym{k} \times \sym{p}}^{\sym{pk}} (\Lie(k) \boxtimes \F)$ in the stable module category (see Proposition \ref{P:just}).


This article is organised as follows: In Section \ref{S:prelim}, we gather together some background knowledge and introduce the notations that we shall need. In Section \ref{S:Heller}, we describe the Heller translates $\Omega^i(\Lie(pk))$ in terms of the Specht and dual Specht modules in the principal block of $\F\sym{p}$ and the projective indecomposable $\F\sym{k}$-modules. As a consequence, we obtain the vertices and sources of the non-projective indecomposable summands of $\Omega^i(\Lie(pk))$. The period of $\Lie(pk)$ is computed in Section \ref{S:period}.

\section{Preliminaries}\label{S:prelim}

In this section, we provide the necessary background, introduce the notations and prove some preliminary results that we shall need.

\subsection{Vertices and sources}

Let $G$ be a finite group, and let $H$ be a subgroup of $G$.  An $\F G$-module $M$ is relatively $H$-projective if and only if $M$ is a direct summand of $\Ind_H^G N$ for some $\F H$-module $N$.  Every $\F G$-module is relatively $P$-projective for any Sylow $p$-subgroup $P$ of $G$, and is projective if and only if it is relatively $\{1\}$-projective.

A vertex $Q$ of an indecomposable $\F G$-module $M$ is a subgroup of $G$ that is minimal subject to $M$ being relatively $Q$-projective. The vertices of $M$ form a single conjugacy class of $p$-subgroups of $G$.  If $M$ has vertex $Q$, with $M$ being a direct summand of $\Ind_Q^G N$ for some indecomposable $\F Q$-module $N$, then we call $N$ a $Q$-source of $M$.

Let $Q$ be a $p$-subgroup of $G$.  Green's correspondence provides a bijection between the set of isomorphism classes of the indecomposable $\F G$-modules with vertex $Q$ and and that of the indecomposable $\F N_G(Q)$-modules with vertex $Q$.  In particular, if $Q$ has order $p$, and $M$ is an indecomposable $\F G$-module with vertex $Q$, and $N$ is its Green correspondent, then
$$
\Res^G_{N_G(Q)} M \cong N \oplus P, \qquad \Ind_{N_G(Q)}^G N \cong M \oplus  R$$
for some projective modules $P$ and $R$ respectively.

\subsection{Periodic modules}

Let $G$ be a finite group. 
Traditionally (see \cite[\S5.10]{B}) an $\F G$-module $M$ is periodic if and only if $\Omega^n(M) \cong M$ for some $n \in \mathbb{Z}^+$, and the least such $n$ for which this is true is the period of $M$.  However, this definition excludes modules which have non-zero projective direct summands.  As the Lie modules have large non-zero projective direct summands (see \cite{BLT} for details), this definition does not suit our purpose here.

In this paper, we say that an $\F G$-module $M$ is {\em periodic} if and only if $\Omega^n(M) \cong \Omega^0(M) \ne 0$ for some $n \in \mathbb{Z}^+$, and the least such $n$ for which this is true is the {\em period} of $M$.  Equivalently, $M$ is periodic if and only if $M$ is non-projective and there exists a long exact sequence of $\F G$-modules \[0\to \Omega^0(M) \to P_m\to \cdots\to P_1\to M\to 0\] where $P_1,\ldots,P_m$ are projectives, and the least such $m \in \mathbb{Z^+}$ for which such a long exact sequence exists is the period of $M$.



We gather together some elementary results about periodic modules.

\begin{lem}\label{L:period}
Let $G$ be a finite group and let $M$ be a periodic $\F G$-module of period $n$. 
\begin{enumerate}
  \item [(i)] Any finite external direct sum of the module $M$ is periodic of period $n$.
  \item [(ii)] $\Omega^aM$ is periodic of period $n$ for all $a \in \mathbb{Z}$.
  \item [(iii)] Let $H$ be a subgroup of $G$. Then $\Res^G_H M$ is either projective or periodic of period dividing $n$.
  \item [(iv)] Let $N$ be a group of which $G$ is a subgroup. Then $\Ind^N_G M$ is periodic of period dividing $n$.
\end{enumerate}
\end{lem}

\begin{proof}
  Parts (i) and (ii) are clear.  For (iii) and (iv), they follow from the fact that $\Res^G_H \Omega^n(M) \cong \Omega^n(\Res^G_H M)  \oplus P$ and $\Ind_G^N \Omega^n(M) \cong \Omega^n(\Ind_G^N M)\oplus Q$, where $P$ and $Q$ are some projective $\F H$- and $\F N$-modules respectively \cite[Proposition 4.4(vii,viii)]{Ca}.
\end{proof}

\subsection{Symmetric groups}

In this subsection, we set up the notations of some elements and subgroups of symmetric groups that shall be used in this paper.

For $n \in \mathbb{Z}^+$, the symmetric group of $n$ letters will be denoted by $\sym{n}$.

Let $a,b \in \mathbb{Z}^+$.

For $\tau \in \sym{a}$, define
$\tau^{[b]} \in \sym{ab}$ by
$$ \tau^{[b]}((i-1)b+  j)  = (\tau(i)-1)b +j
$$
for all $1 \leq i \leq a$ and $1\leq j \leq b$, so that $\tau^{[b]}$
permutes the $a$ successive blocks of size $b$ in $\{1, \dotsc, ab\}$
according to $\tau$.
Clearly, the map $\tau \mapsto \tau^{[b]}$ is an injective group homomorphism.

For $\sigma \in \sym{b}$ and $1\leq r \leq a$, define $\sigma[r] \in \sym{ab}$ by
$$
(\sigma[r])(i+ (j-1)a)  =
\begin{cases}
\sigma(i)+(r-1)a &\text{if } j = r, \\
i+(j-1)a  &\text{otherwise},
\end{cases}
$$
for all $1 \leq i \leq a$ and $1 \leq j \leq b$, so that $\sigma[r]$ acts on the $r$th successive block of size $b$ in $\{1,\dotsc, ab\}$ according to $\sigma$, and fixes everything else. In addition, define $\Delta_a (\sigma) \in \sym{ab}$ by $\Delta_a (\sigma) = \prod_{r=1}^{a} \sigma[r]$, so that $\Delta_a (\sigma)$ permutes each of the $a$ successive blocks of size $b$ in $\{1,\dotsc, ab\}$ simultaneously according to $\sigma$.  Clearly, the maps $\sigma \mapsto \sigma[r]$ and $\sigma \mapsto \Delta_a(\sigma)$ are again injective group homomorphisms.

If $H \subseteq \sym{a}$, $L \subseteq \sym{b}$ and $r \in \{1,\dotsc, a\}$, we write
\begin{align*}
H^{[b]} &= \{ h^{[b]}\ :\ h \in H \} ,\\
L[r] &= \{ \ell[r]\ :\  \ell \in L \} ,\\
\Delta_a (L) &= \{ \Delta_a(\ell)\ :\ \ell \in L \}.
\end{align*}
Note that $\Delta_a(\sym{b})$ and $\sym{a}^{[b]}$ centralises each other.


\subsection{Some miscellaneous results}

In this subsection, we state and prove three results which we shall require.  These results are most probably known and their proofs are provided here for the sake of completeness.

For any finite group $G$, we write $\F_G$ for the trivial $\F G$-module and subsequently drop the subscript $G$ if it is clear from the context.

\begin{lem} \label{L:coprimesummand}
Let $H$ be a subgroup of a group $G$ of index $\ell$ with $p \nmid \ell$.  Denote the trivial $\F G$- and $\F H$-module by $\F_G$ and $\F_H$ respectively.  Then $\F_G$ is a direct summand of $\Ind_H^G \F_H = \F G \otimes_{\F H} \F_H$ as an $\F G$-module.
\end{lem}

\begin{proof}
Let $T$ be a complete set of left coset representatives of $H$ in $G$.
It is straightforward to check that the maps
\begin{alignat*}{2}
\phi : \F_G &\to \Ind_H^G \F_H  \qquad& \psi: \Ind_H^G \F_H &\to \F_G \\
\phi(\lambda) &= \frac{1}{\ell}\sum_{t \in T} t \otimes \lambda \qquad \qquad & \psi\left (\sum_{t \in T} t \otimes \mu_t\right ) &= \sum_{t \in T} \mu_t
\end{alignat*}
are $\F G$-module homomorphisms satisfying $\psi \circ \phi = \operatorname{id}_{\F_G}$.  Thus $\phi$ is injective, and $\Ind_H^G \F_H = \ker(\psi) \oplus \phi(\F_G)$.
\end{proof}

\begin{lem} \label{L:F}
Let $\F$ be the trivial $\F\sym{p}$-module.
\begin{enumerate}
\item [(i)] For $p=2$, $\Omega(\F)=\F$.
\item [(ii)] For $p\geq 3$ and $i\equiv j\pmod{2p-2}$ where $j\in \{0,1,\ldots,2p-3\}$, we have \[\Omega^i(\F)=\left \{\begin{array}{ll} S^{(p-j,1^j)}&\text{if $0\leq j\leq p-1$,}\\ S_{(j-p+2,1^{2p-j-2})}&\text{if $p\leq j\leq 2p-3$}.\end{array}\right .\]
\end{enumerate} In particular, the trivial $\F\sym{p}$-module $\F$ has period $\begin{cases} 2p-2, &\text{if $p$ is odd}, \\ 1, &\text{if } p=2.\end{cases}$
\end{lem}

\begin{proof}
For $p=2$, $\F\sym{p}$ is projective indecomposable and so $\Omega(\F) = \F$.

For $p\geq 3$ and $j=0,1,\ldots,p-1$, the Specht modules $S^{(p-j,1^j)}$ are pairwise non-isomorphic, and similarly, the dual Specht modules $S_{(p-j,1^j)}\cong (S^{(p-j,1^j)})^*$ are also pairwise non-isomorphic.  Furthermore, $S^{(p-j,1^j)} \cong S_{(p-\ell,1^\ell)}$ if and only if $j= \ell \in \{0,p-1\}$.  It is well-known that
$\Omega(S^{(p-j,1^j)}) = S^{(p-j-1,1^{j+1})}$ for $0 \leq j \leq p-2$ while $\Omega(S_{(p-\ell,1^\ell)}) = S_{(p-\ell+1,1^{\ell-1})}$ for $1 \leq \ell \leq p-1$. Thus we obtain part (ii). The final assertion is now clear.
\end{proof}


\begin{prop} \label{P:Ind F}
Suppose that $p$ is odd, and let $C_p$ be a subgroup of $\sym{p}$ of order $p$.  Then $$\Ind_{C_p}^{\sym{p}} \F \cong \F \oplus \Omega^2(\F) \oplus \dotsb \oplus\Omega^{2p-4}(\F) \oplus P$$ for some projective $\F\sym{p}$-module $P$.
\end{prop}

\begin{proof}
In this proof, for a subgroup $H$ of $\sym{p}$, we denote the trivial $\F H$-module as $\F_H$.

Let $L = N_{\sym{p}}(C_p)$.  Since $\Res^{\sym{p}}_{L} \F_{\sym{p}} = \F_L$, and $\F_{\sym{p}}$ has vertex $C_p$, we see that $\F_{\sym{p}}$ and $\F_L$ are Green correspondents of each other.  Thus $\Omega^i(\F_{\sym{p}})$ and $\Omega^i(\F_L)$ are Green correspondents of each other for all $i \in \mathbb{Z}^+$ (see \cite[\S20, Proposition 7]{A}).  In particular, since Green correspondence is bijective, this shows that $\F_L$ is periodic, having the same period as $\F_{\sym{p}}$, i.e. $2p-2$ by Lemma \ref{L:F}(ii).

It is well-known that the trivial $\F C_p$-module $\F_{C_p}$ is periodic of period $2$.  Thus $\Ind_{C_p}^L \F_{C_p}$ is periodic of period $1$ or $2$ by Lemma \ref{L:period}(iv).  Since $\F_{L}$ is a direct summand of $\Ind_{C_p}^L (\F_{C_p})$ by Lemma \ref{L:coprimesummand}, this implies that $\Omega^{2i}(\F_{L})$ is a direct summand of $\Ind_{C_p}^L(\F_{C_p})$ for all $i \in \mathbb{Z}$.  Since $\F_L$ has period $2p-2$, and each $\Omega^{2i}(\F_L)$ is indecomposable (see \cite[\S20, Theorem 5(1)]{A}), we see that $\F_L \oplus \Omega^2(\F_L) \oplus \dotsb \oplus \Omega^{2p-4}(\F_L)$ is a direct summand of $\Ind_{C_p}^L \F_{C_p}$.  Thus
$$
p-1 \leq \dim_{\F} ( \F_L \oplus \Omega^2(\F_L) \oplus \dotsb \oplus \Omega^{2p-4}(\F_L)) \leq \dim_{\F} (\Ind_{C_p}^L \F_{C_p}) = p-1$$
(the last equality follows from the fact that $L = N_{\sym{p}}(C_p)$ is a semidirect product of $C_p$ with a cyclic group of order $p-1$), forcing equality throughout.  Hence
$$
\Ind_{C_p}^L \F_{C_p} \cong \F_L \oplus \Omega^2(\F_L) \oplus \dotsb \oplus \Omega^{2p-4}(\F_L).
$$
Since $\Omega^i(\F_{\sym{p}})$ is the Green correspondent of $\Omega^i(\F_L)$, we get
\begin{align*}
\Ind_{C_p}^{\sym{p}} \F_{C_p} = \Ind_{L}^{\sym{p}} (\Ind_{C_p}^L \F_{C_p}) &= \Ind_L^{\sym{p}} (\F_L \oplus \Omega^2(\F_L) \oplus \dotsb \oplus \Omega^{2p-4}(\F_L)) \\
&\cong \F_{\sym{p}} \oplus \Omega^2(\F_{\sym{p}}) \oplus \dotsb \oplus \Omega^{2p-4}(\F_{\sym{p}}) \oplus P,
\end{align*}
where $P$ is a projective $\F \sym{p}$-module.
\end{proof}

\section{Heller translates of $\Lie(pk)$}\label{S:Heller}


In this section, we describe the Heller translates of the periodic Lie module $\Lie(pk)$ when $p\nmid k$ (see Theorem \ref{T: Heller}). Our strategy is based on an analysis of the $p$th symmetrization $S^p(\Lie(k))$ of $\Lie(k)$ that is defined in \cite{ES}. We begin with a lemma.

\begin{lem} \label{L:Lie(k)}
Let $k \in \mathbb{Z}^+$ with $p \nmid k$, and let $C_k$ be a cyclic subgroup of $\sym{k}$ generated by a $k$-cycle.  Then $\Lie(k) \cong \Ind_{C_k}^{\sym{k}} \F_{\delta}$ for any faithful one-dimensional $\F C_k$-module $\F_{\delta}$.
\end{lem}

\begin{proof}
This was proved in \cite[Corollary 8.7]{R} when $\F$ has characteristic zero using the argument of ordinary characters. In prime characteristic, since both modules are projective (as is well-known), they are liftable and thus have equal ordinary characters.  Hence these modules are equal in the Grothendieck group $K_0(\F\sym{k})$ and hence isomorphic since the set of indecomposable projective modules is a $\mathbb{Z}$-basis for $K_0(\F\sym{k})$.

Alternatively, Klyachko \cite[Theorem]{K} proved that the modular Lie power $L^k(V)$ is isomorphic to $V^{\otimes k} \otimes_{\F C_k} \F_{\delta}$ as left modules for the Schur algebra $S(n,k)$, where $V$ is a $\F$-vector space of dimension $n$ and $V^{\otimes k}$ is a $(S(n,k), \F\sym{k})$-bimodule with $\sym{k}$ acting on the right by place permutations.  Applying the Schur functor when $n=k$, which sends $L^k(V)$ to $\Lie(k)$ and $V^{\otimes k}$ to $\F\sym{k}$, we get the required isomorphism.
\end{proof}

The $p$th symmetrization of $\Lie(k)$, denoted as $S^p(\Lie(k))$, is first introduced in \cite{ES} to study the Lie module $\Lie(pk)$.  It is an $\F\sym{pk}$-module, and may be defined by
$$
S^p(\Lie(k)) = \Ind_{\sym{k} \wr  \sym{p}}^{\sym{pk}} (\Lie(k)^{\otimes p})
$$
where $\sym{p}$ acts on $\Lie(k)^{\otimes p}$ by place permutation, and $\sym{k} \wr \sym{p}$ is identified with the subgroup $(\prod_{i=1}^p \sym{k}[i])\sym{p}^{[k]}$ of $\sym{pk}$.
This module is instrumental in the study of $\Lie(pk)$ because of the following result:

\begin{thm}[{\cite[Theorem 10]{ES}}]\label{T: Erdmann-Schocker}
Let $k \in \mathbb{Z}^+$ with $p \nmid k$.  There is a short exact sequence of left $\F\sym{pk}$-modules
$$
0 \to \Lie(pk) \to \F\sym{pk}\,e \to S^p(\Lie(k)) \to 0,
$$
where $e$ is an idempotent of $\F\sym{pk}$.
\end{thm}

The module $S^p(\Lie(k))$ is studied in detail in \cite{ES} in a more general setting as $S^p(U)$ where $U$ is any projective $\F\sym{k}$-module.  In particular, we have the following:

\begin{prop} \label{P:just}
Let $C_p$ be a fixed $p$-subgroup of $\sym{p}$.  Then
$$\Omega^0(S^p(\Lie(k))) \cong \Omega^0(\Ind_{\sym{k} \times N_{\sym{p}}(C_p)}^{\sym{pk}} (\Lie(k) \boxtimes \F)) \cong \Omega^0(\Ind_{\sym{k} \times \sym{p}}^{\sym{pk}} (\Lie(k) \boxtimes \F))$$
where we identify $\sym{k} \times N_{\sym{p}}(C_p)$ and $\sym{k} \times \sym{p}$ with the subgroups $\Delta_p(\sym{k}) (N_{\sym{p}}(C_p))^{[k]}$ and $\Delta_p(\sym{k})\sym{p}^{[k]}$ of $\sym{pk}$ respectively.
\end{prop}

\begin{proof}
Let $M = \Lie(k)^{\otimes p}$, an $\F(\sym{k} \wr \sym{p})$-module, so that $S^p(\Lie(k)) = \Ind_{\sym{k}\wr \sym{p}}^{\sym{pk}} M$, where $\sym{k} \wr \sym{p}$ is identified with the subgroup $(\prod_{i=1}^p \sym{k}[i])\sym{p}^{[k]}$. Then
$$
M \cong (\Ind^{\sym{k}}_{C_k} \F_{\delta})^{\otimes p} \cong \Ind_{C_k \wr \sym{p}}^{\sym{k} \wr \sym{p}}  \left ((\F_{\delta})^{\otimes p}\right )$$
by Lemma \ref{L:Lie(k)} and \cite[Lemma 5]{ELT}.  In particular, $M$ is relatively $C_p$-projective since $C_p$ is the Sylow $p$-subgroup of $C_k \wr \sym{p}$.  Note also that $M$ is not projective since its dimension is not divisible by the order of a Sylow $p$-subgroup of $\sym{k} \wr \sym{p}$.

Let $$M = \left (\bigoplus_{i=1}^\ell X_i\right ) \oplus P_H,$$ where each $X_i$ is non-projective indecomposable and $P_H$ is projective.  Then each $X_i$ is relatively $C_p$-projective (since $M$ is) and hence it has a vertex $C_p$.  In addition, by \cite[proof of Lemma 13]{ES}, we also have $$\Ind_{\sym{k} \wr \sym{p}}^{\sym{pk}} X_i = Y_i \oplus Q_i$$ for some indecomposable non-projective $\F\sym{pk}$-module $Y_i$ and projective $\F\sym{pk}$-module $Q_i$.

By \cite[proof of Proposition 12]{ES}, $M$ is a summand of a permutation module .  In particular its direct summands $X_i$ are all $p$-permutation modules.  Thus each Brauer quotient $X_i(C_p)$, as an $\F N_{\sym{k}\wr \sym{p}}(C_p)$-module, is the Green correspondent of $X_i$ \cite[Ex.\ 27.4]{T}, so that $$\Ind_{N_{\sym{k} \wr \sym{p}}(C_p)}^{\sym{k} \wr \sym{p}} X_i(C_p) \cong X_i \oplus R_i$$ for some projective $\F (\sym{k} \wr \sym{p})$-module $R_i$.  Consequently, identifying an element $(\sigma_1,\dotsc,\sigma_p)\tau$ of $\sym{k} \wr \sym{p}$ with $(\prod_{i=1}^p \sigma_i[i])\tau^{[k]} \in \left(\prod_{i=1}^p \sym{k}[i]\right)\sym{p}^{[k]} \subseteq \sym{pk}$, we have
\begin{align*}
\Ind_{N_{\sym{k}\wr\sym{p}}(C_p)}^{\sym{pk}} M(C_p) &= \Ind_{N_{\sym{k} \wr\sym{p}}(C_p)}^{\sym{pk}} \left(\bigoplus_{i=1}^\ell X_i(C_p) \right) \\
&\cong \Ind_{\sym{k} \wr \sym{p}}^{\sym{pk}} \left(\bigoplus_{i=1}^\ell (X_i \oplus R_i) \right) = \left (\bigoplus_{i=1}^\ell Y_i\right ) \oplus P
\end{align*}
for some projective $\F\sym{pk}$-module $P$, so that
$$\Omega^0\left(\Ind_{N_{\sym{k}\wr\sym{p}}(C_p)}^{\sym{pk}} M(C_p)\right) \cong \bigoplus_{i=1}^\ell Y_i \cong \Omega^0\left(\Ind_{\sym{k} \wr \sym{p}}^{\sym{pk}} M\right) = \Omega^0\left(S^p(\Lie(k))\right).$$
The first isomorphism now follows, since
$$
N_{\sym{k}\wr\sym{p}}(C_p) = \sym{k} \times N_{\sym{p}}(C_p)$$
and $M(C_p)$ is isomorphic to $\Lie(k) \boxtimes \F$ as $\F (\sym{k} \times N_{\sym{p}}(C_p))$-modules \cite[between Lemma 14 and 15, and Lemma 15]{ES}.

For the second isomorphism, note first that the trivial $\F N_{\sym{p}}(C_p)$-module is the Green correspondent of the trivial $\F \sym{p}$-module by Lemma \ref{L:coprimesummand}.  Thus, for each projective indecomposable $\F \sym{k}$-module $P$, the $\F(\sym{k} \times N_{\sym{p}}(C_p))$-module $P \boxtimes \F$ is the Green correspondent of the $\F(\sym{k} \times \sym{p})$-module $P \boxtimes \F$.  Since $\Lie(k)$ is projective as an $\F \sym{k}$-module, this shows that
$$
\Omega^0\left(\Ind_{\sym{k} \times N_{\sym{p}}(C_p)}^{\sym{k} \times \sym{p}} (\Lie(k) \boxtimes \F)\right) = \Lie(k) \boxtimes \F ,$$
and the second asserted isomorphism follows.
\end{proof}

\begin{cor} \label{C:use}
Let $k \in \mathbb{Z}^+$ with $p \nmid k$.  Then
$$
\Omega^0(\Lie(pk)) \cong \Omega\left(\Ind_{\sym{p} \times \sym{k}}^{\sym{pk}} (\F \boxtimes \Lie(k))\right)
$$
where we identify $\sym{p} \times \sym{k}$ with the subgroup $\Delta_k(\sym{p}) \sym{k}^{[p]}$ of $\sym{pk}$.
\end{cor}

\begin{proof}
By Theorem \ref{T: Erdmann-Schocker} and Proposition \ref{P:just},
$$
\Omega^0(\Lie(pk)) \cong \Omega\left(\Ind_{\sym{k} \times \sym{p}}^{\sym{pk}} (\Lie(k) \boxtimes \F)\right)
$$
where $\sym{k} \times \sym{p}$ is identified with $\Delta_p(\sym{k})\sym{p}^{[k]}$.  But $\Ind_{\sym{k} \times \sym{p}}^{\sym{pk}} (\Lie(k) \boxtimes \F)$ is isomorphic to $\Ind_{\sym{p} \times \sym{k}}^{\sym{pk}} (\F \boxtimes \Lie(k))$, where $\sym{k} \times \sym{p}$ is identified with $\Delta_p(\sym{k})\sym{p}^{[k]}$ in the former and $\sym{p} \times \sym{k}$ is identified with $\Delta_k(\sym{p})\sym{k}^{[p]}$ in the latter, via the map sending $g \otimes (x \otimes 1) \mapsto gz^{-1} \otimes (1 \otimes x)$ for $g \in \sym{pk}$ and $x \in \Lie(k)$, where $z \in \sym{pk}$ such that $z\left (\Delta_p(\sym{k})\sym{p}^{[k]}\right )z^{-1} = \Delta_k(\sym{p})\sym{k}^{[p]}$.  The corollary thus follows.
\end{proof}

\begin{rem}
In \cite{ET}, $S^p(\Lie(k))$ was defined as $\Ind_{\sym{p} \times \sym{k}}^{\sym{pk}} (\F \boxtimes \Lie(k))$ (where $\sym{p} \times \sym{k}$ is identified with the subgroup $\Delta_k(\sym{p}) \sym{k}^{[p]}$ of $\sym{pk}$).  This suited their purposes and they remarked that their definition is the `same' as that used in \cite{ES} without giving any justification.  Corollary \ref{C:use} now provides the omitted details.
\end{rem}

We note that part (i) of Theorem \ref{T:main} now follows from Corollary \ref{C:use} and \cite[Proposition 4.4(viii)]{Ca}.  In fact, we are able to give a more precise description of the Heller translates of $\Lie(pk)$ as follows:






\begin{thm}\label{T: Heller} Let $p\nmid k$ and let $\Lie(k) = \bigoplus_{s=1}^{m_k}P_s$ be a decomposition of $\Lie(k)$ into its indecomposable (projective) summands.  We identify $\sym{p}\times \sym{k}$ with the subgroup $\Delta_k(\sym{p}) \sym{k}^{[p]}$ of $\sym{pk}$.
\begin{enumerate}

\item [(i)] If $p=2$ then, for all $i\in\Z$, \[\Omega^i(\Lie(2k))\cong \bigoplus_{s=1}^{m_k} \Omega^0(\Ind_{\sym{2}\times \sym{k}}^{\sym{2k}}(\F\boxtimes P_s)).\] Furthermore, for each $s\in \{1,\ldots,m_k\}$, the $\F\sym{2k}$-module $\Omega^0(\Ind_{\sym{2}\times \sym{k}}^{\sym{2k}}(\F\boxtimes P_s))$ is indecomposable.
\item [(ii)] If $p$ is odd and $i\equiv j\pmod{2p-2}$ where $j\in \{0,1,\ldots,2p-3\}$, then
$$
\Omega^i(\Lie(pk))\cong
\begin{cases}
\bigoplus_{s=1}^{m_k} \Omega^0(\Ind_{\sym{p}\times \sym{k}}^{\sym{pk}}(S^{(p-j-1,1^{j+1})}\boxtimes P_s)), &\text{if } 0 \leq j\leq p-2;\\
\bigoplus_{s=1}^{m_k} \Omega^0(\Ind_{\sym{p}\times \sym{k}}^{\sym{pk}}(S_{(j-p+3,1^{2p-j-3})}\boxtimes P_s)), &\text{if } p-1 \leq j \leq 2p-3.
\end{cases}
$$
Furthermore, for any admissible $j$, $\ell$ and $s$, the $\F\sym{pk}$-modules $$\Omega^0(\Ind_{\sym{p}\times \sym{k}}^{\sym{pk}}(S^{(p-j-1,1^{j+1})}\boxtimes P_s))\ \text{and} \ \Omega^0(\Ind_{\sym{p}\times \sym{k}}^{\sym{pk}}(S_{(\ell-p+3,1^{2p-\ell-3})}\boxtimes P_s))$$ are indecomposable.
\end{enumerate}
\end{thm}

\begin{proof} Since $\Lie(k) = \bigoplus_{s=1}^{m_k}P_s$, we have, by Corollary \ref{C:use} and \cite[Proposition 4.4(viii)]{Ca},
\begin{align*}
\Omega^i(\Lie(pk))&\cong \Omega^{i+1}\left(\Ind_{\sym{p} \times \sym{k}}^{\sym{pk}} (\F \boxtimes \Lie(k))\right)\\
&= \Omega^{i+1}\left(\bigoplus_{s=1}^{m_k}\Ind_{\sym{p} \times \sym{k}}^{\sym{pk}} (\F \boxtimes P_s)\right)\\
&\cong \bigoplus_{s=1}^{m_k}\Omega^0(\Ind_{\sym{p} \times \sym{k}}^{\sym{pk}} (\Omega^{i+1}(\F) \boxtimes P_s)).
\end{align*} Now the first assertions of both parts (i) and (ii) follow from Lemma \ref{L:F}.

By \cite[Theorem 2]{ES}, there is a one-to-one multiplicity-preserving correspondence between the non-projective indecomposable summands of $\Lie(pk)$ and the indecomposable summands of the projective $\F\sym{k}$-module $\Lie(k)$.  This implies that $\Lie(pk)$, and hence $\Omega^i(\Lie(pk))$ for any $i \in \Z$, each has exactly $m_k$ non-projective indecomposable summands in its decomposition into indecomposable summands.

Since $\Ind_{\sym{p} \times \sym{k}}^{\sym{pk}} (\Omega^{i+1}(\F) \boxtimes P_s)$ is non-projective (as $\Omega^{i+1}(\F)$ is non-projective), we see that $\Omega^0(\Ind_{\sym{p} \times \sym{k}}^{\sym{pk}} (\Omega^{i+1}(\F) \boxtimes P_s))\ne 0$ for all $s\in\{1,\ldots,m_k\}$. Since $\Omega^i(\Lie(pk))$ has exactly $m_k$ non-projective summands, it follows that all $\Omega^0(\Ind_{\sym{p} \times \sym{k}}^{\sym{pk}} (\Omega^{i+1}(\F) \boxtimes P_s))$ are indecomposable.
\end{proof}

We remarked that the indecomposable summands of $\Lie(k)$ have been described in \cite[\S3.3, Theorem]{SDKE}.

Theorem \ref{T: Heller}(i) implies that $\Lie(2k)$ has period $1$ when $p=2$. 
When $p\geq 3$, we may conclude directly (and independently of the results in \cite{ELT}) that $\Lie(pk)$ is periodic, and its period divides $2p-2$.  
However, it is not immediately clear that its period is exactly $2p-2$, as we have yet to determine if it is possible for $\Omega^i(\Lie(pk)) \cong \Omega^j(\Lie(pk))$ when $0\leq i,j\leq 2p-3$ with $i \ne j$. 
The following corollary shows that $\Omega^0(\Ind_{\sym{p} \times \sym{k}}^{\sym{pk}} (\Omega^{\ell}(\F) \boxtimes P_s))$ share a common vertex and fairly common sources, lending some evidence that such an undesirable isomorphism may occur. 

\begin{cor}\label{C:vertexsource} For all $i\in\Z$, any indecomposable non-projective summand of the $\F\sym{pk}$-module $\Omega^i(\Lie(pk))$ has a vertex $\Delta_k(C_p)$, and $\Delta_k(C_p)$-source $\F$ when $i$ is odd and $\Omega(\F)$ when $i$ is even.
\end{cor}

\begin{proof} 
By Theorem \ref{T: Heller}, we may assume that $i=0$ if $p=2$ and $i=0,1,\ldots,2p-3$ if $p\geq 3$.  In addition, any indecomposable non-projective summand of $\Omega^i(\Lie(pk))$ is of the form $\Omega^0(\Ind_{\sym{p}\times \sym{k}}^{\sym{pk}}(\Omega^{i+1}(\F)\boxtimes P_s))$ for some $s \in \{1,\ldots,m_k\}$. Let $\ell=i+1$. By Proposition \ref{P:Ind F}, $\Omega^{\ell}(\F_{\sym{p}})$ has a vertex $C_p$ and $C_p$-source $\F$ if $\ell$ is even. Also, we have \[\Omega^0\left (\Ind_{C_p}^{\sym{p}}\Omega(\F)\right )\cong \Omega\left (\Ind_{C_p}^{\sym{p}} \F\right ) \cong \Omega(\F) \oplus \Omega^3(\F) \oplus \dotsb \oplus\Omega^{2p-3}(\F).\] This shows that $\Omega^{\ell}(\F_{\sym{p}})$ has a vertex $C_p$ and $C_p$-source $\Omega(\F)$ if $\ell$ is odd. Let $S_\ell$ be the $C_p$-source of $\Omega^\ell(\F)$. Since $P_s$ is projective, and hence injective, we see that $P_s$ is a direct summand of $\F\sym{k}$.
Thus we obtain $\Omega^0(\Ind_{\sym{p}\times \sym{k}}^{\sym{pk}}(\Omega^\ell(\F)\boxtimes P_s))$ is a direct summand of $\Ind_{\sym{p}\times \sym{k}}^{\sym{pk}}\left ((\Ind_{C_p}^{\sym{p}}S_\ell)\boxtimes \F\sym{k}\right )\cong \Ind_{\Delta_k(C_p)}^{\sym{pk}}S_\ell$. This completes our proof.
\end{proof}

\begin{rem} 
In \cite[para.\ before Lemma 13]{ES}, it is shown that $C_p^{[k]}$ is an vertex of any non-projective indecomposable summand of $S^p(U)$, where $U$ is a projective $\F\sym{k}$-module. Since vertices are invariant under Heller translates, we may alternatively obtain $C_p^{[k]}$ as an vertex of any indecomposable non-projective summand of $\Omega^i(\Lie(pk))$ at once by specialising this result to the case when $U=\Lie(k)$.
\end{rem}

We end this section with a result which may be of independent interest.

\begin{cor} The Lie module $\Lie(pk)$ is an endo-$p$-permutation module.
\end{cor}
\begin{proof} By Corollary \ref{C:vertexsource} , since $\Omega^{-1}(\Lie(pk))$ is a direct sum of trivial source modules, it is a $p$-permutation module and hence an endo-$p$-permutation module. By definition, $\Lie(pk)\cong \Omega(\Omega^{-1}(\Lie(pk)))\oplus Q$ for some projective $\F\sym{pk}$-module $Q$. Since the class of endo-$p$-permutation is closed under taking Heller translations and taking direct sums with projective modules, it follows that $\Lie(pk)$ 
is endo-$p$-permutation. 
\end{proof}

\section{The period of $\Lie(pk)$}\label{S:period}

In this section, we prove Theorem \ref{T:main}(ii).  Throughout this section, fix $k \in \mathbb{Z}^+$ with $p \nmid k$, $C_p$ and $C_k$ denote fixed cyclic subgroups of $\sym{p}$ and $\sym{k}$ generated by a $p$-cycle and a $k$-cycle respectively, and $\F_{\delta}$ denotes
a fixed faithful one-dimensional $\F C_k$-module.  Also, let $\Delta = \Delta_k(\sym{p})$ and $D = \Delta_k(\sym{p}) C_k^{[p]} $ ($= \sym{p} \times C_k$).

By Theorem \ref{T:main}(i) and Lemma \ref{L:F}, we already knew that $\Lie(pk)$ has period 1 when $p=2$, and that $\Lie(pk)$ has period dividing $2p-2$ when $p \geq 3$, with 
$$\Omega^i(\Lie(pk)) \cong \Omega^{i+1}\left(\Ind_{\sym{p} \times \sym{k}}^{\sym{pk}} \left(\F \boxtimes \Lie(k)\right) \right) \cong \Omega^{i+1}\left(\Ind_{\sym{p} \times C_k}^{\sym{pk}} \left(\F \boxtimes \F_{\delta} \right)\right),$$
where the last isomorphism is given by Lemma \ref{L:Lie(k)}. To complete the proof of Theorem \ref{T:main}(ii), we proceed with a careful analysis of the module 
$$\Lambda_k := \Ind_{\sym{p} \times C_k}^{\sym{pk}} \left(\F \boxtimes \F_{\delta} \right) = \Ind_{D}^{\sym{pk}} \left(\F \boxtimes \F_{\delta} \right).$$

By Mackey decomposition,
$$
\Res^{\sym{pk}}_{\Delta_k(\sym{p})} \Lambda_k \cong \Res^{\sym{pk}}_{\Delta} \Ind_{D}^{\sym{pk}} (\F \boxtimes \F_{\delta}) \cong \bigoplus_{x \in \Delta \setminus \sym{pk} /D} \Ind_{xDx^{-1} \cap \Delta}^{\Delta}\, (x\otimes(\F \boxtimes \F_{\delta})).
$$

Clearly, if $p \nmid |xDx^{-1} \cap \Delta|$, then $x \otimes (\F \boxtimes \F_{\delta})$ is projective as an $\F (xDx^{-1} \cap \Delta)$-module, and hence $\Ind_{xDx^{-1} \cap \Delta}^{\Delta}\, (x\otimes(\F \boxtimes \F_{\delta}))$ is projective.  It remains to describe $\Ind_{xDx^{-1} \cap \Delta}^{\Delta}\, (x\otimes(\F \boxtimes \F_{\delta}))$ when $p \mid |xDx^{-1} \cap \Delta|$.
We begin with a preliminary result.


\begin{prop} \label{P:choice}
Let $x \in \sym{pk}$.
\begin{enumerate}
\item[(i)] If $p \mid |xDx^{-1} \cap \Delta|$, then there exists $x_0 \in (\prod_{r=1}^k C_p[r])\sym{k}^{[p]}$ such that $\Delta x_0 D = \Delta x D$.
\item[(ii)] If $x \in (\prod_{r=1}^k C_p[r])\sym{k}^{[p]}$, say $x= yz$ where $y \in \prod_{r=1}^k C_p[r]$ and $z \in \sym{k}^{[p]}$, then $$xDx^{-1} \cap \Delta = y\Delta y^{-1} \cap \Delta =
\begin{cases}
\Delta, &\text{if } y \in \Delta_k(C_p), \\
\Delta_k(C_p), &\text{otherwise.}
\end{cases}
$$
\end{enumerate}
\end{prop}

\begin{proof} \hfill
\begin{enumerate}
\item[(i)] Let $a \in xDx^{-1} \cap \Delta$ be an element of order $p$.  Then $a = \Delta_k(\sigma) = x\Delta_k(\rho)\tau^{[p]}x^{-1}$ for some $\sigma, \rho \in \sym{p}$ and $\tau \in C_k$.  Since $a$ has order $p$, while the order of $x\Delta_k(\rho)\tau^{[p]}x^{-1}$ is the least common multiple of those of $\rho$ and $\tau$, and $\tau$ has order coprime to $p$, we see that $\tau = 1$ and $\rho$ has order $p$.  Let $\pi$ be a generator of $C_p$.  Then $\pi$, $\sigma$ and $\rho$ are $p$-cycles in $\sym{p}$, so there exist $\xi, \eta \in \sym{p}$ such that $\sigma = \xi\pi\xi^{-1}$ and $\rho =\eta \pi \eta^{-1}$.  Thus
$$
\Delta_k(\xi) \Delta_k(\pi) \Delta_k(\xi)^{-1} = \Delta_k(\sigma) = x \Delta_k(\rho) x^{-1} =  x \Delta_k(\eta) \Delta_k(\pi) \Delta_k(\eta)^{-1} x^{-1}.
$$
Let $x_0 = \Delta_k(\xi^{-1}) x \Delta_k(\eta)$.  Then $\Delta x_0 D = \Delta x D$, and
$\Delta_k(\pi) = x_0 \Delta_k(\pi) x_0^{-1}$, so that $x_0 \in C_{\sym{pk}}(\Delta_k(C_p)) = (\prod_{r=1}^k C_p[r]) \sym{k}^{[p]}$.
\item[(ii)] Let $a \in xDx^{-1} \cap \Delta$.  Then $a = \Delta_k(\sigma) = x\Delta_k(\rho)\tau^{[p]}x^{-1}$ for some $\sigma, \rho \in \sym{p}$ and $\tau \in C_k$.  Since $z$ commutes with $\Delta_k(\rho)$, we have
    $$
    \prod_{r=1}^k \sym{p}[r] \ni \Delta_k(\rho)^{-1} y^{-1}\Delta_k(\sigma) y = z \tau^{[p]} z^{-1} \in \sym{k}^{[p]},$$
    so that $\tau = 1$ and $a= \Delta_k(\sigma) = y \Delta_k(\rho) y^{-1}$.  Thus $a \in \Delta \cap y\Delta y^{-1}$.        Conversely, $x \Delta x^{-1} = y \Delta y^{-1}$ since $z$ centralises $\Delta$, so that
    $$
    y \Delta y^{-1} \cap \Delta = x\Delta x^{-1} \cap \Delta \subseteq xDx^{-1} \cap \Delta.$$
    This proves the first equality.

    Now we prove the second equality.  If $y \in \Delta_k(C_p)$, then $$y\Delta y^{-1} \cap \Delta = \Delta \cap \Delta = \Delta.$$  Assume thus $y \notin \Delta_k(C_p)$.  Let $y = \prod_{r=1}^k y_r[r]$ where $y_r \in C_p$ for all $r$.  By assumption, $y_i \ne y_j$ for some $i$ and $j$. Let $a \in y\Delta y^{-1} \cap \Delta$.  Then $a = \Delta_k(\sigma) = y\Delta_k(\rho)y^{-1}$ for some $\sigma, \rho \in \sym{p}$.  Thus $y_i \rho y_i^{-1} = \sigma = y_j\rho y_j^{-1}$, so that $\rho \in C_{\sym{p}}(y_i^{-1}y_j) = C_p$.  Hence, $a = y \Delta_k(\rho)y^{-1} = \Delta_k(\rho) \in \Delta_k(C_p)$.  Consequently, $y\Delta y^{-1} \cap \Delta \subseteq \Delta_k(C_p)$.  The reverse inclusion is clear.
\end{enumerate}
\end{proof}

\begin{thm} \label{T:Res}
Let
\begin{align*}
 \Gamma_1 &= \{ x \in \Delta \!\setminus\! \sym{pk}/ D\ :\  xDx^{-1} \cap \Delta = \Delta \}, \\
 \Gamma_2 &= \{ x \in \Delta \!\setminus\! \sym{pk}/ D\ :\  |xDx^{-1} \cap \Delta| = p \}.
\end{align*}
  Then
\begin{align*}
\Res_{\Delta_k(\sym{p})}^{\sym{pk}} \Lambda_k &\cong \F^{\oplus |\Gamma_1|} \oplus (\Ind_{C_p}^{\sym{p}} \F)^{\oplus |\Gamma_2|} \oplus P\\
&\cong \F^{\oplus (|\Gamma_1|+|\Gamma_2|)}\oplus \Omega^2(\F)^{\oplus |\Gamma_2|}\oplus \Omega^4(\F)^{\oplus |\Gamma_2|}\oplus\cdots\oplus \Omega^{2p-4}(\F)^{\oplus |\Gamma_2|}\oplus P
\end{align*}
for some projective $\F\Delta_k(\sym{p})$-module $P$.  Here, $\Ind_{C_p}^{\sym{p}} \F$ is an $\F\Delta_k(\sym{p})$-module via the isomorphism $\sym{p} \cong \Delta_k(\sym{p})$.
\end{thm}

\begin{proof}
By Mackey decomposition,
$$
\Res^{\sym{pk}}_{\Delta} \Lambda_k \cong 
\bigoplus_{x \in \Delta \setminus \sym{pk} /D} \Ind_{xDx^{-1} \cap \Delta}^{\Delta}\, (x \otimes (\F \boxtimes \F_{\delta})).
$$
We have already seen that $\Ind_{xDx^{-1} \cap \Delta}^{\Delta} (x\otimes (\F \boxtimes \F_{\delta}))$ is projective if $p \nmid |xDx^{-1} \cap \Delta|$.  Note that the statement of our theorem involves only the cardinality of $\Gamma_1$ and $\Gamma_2$ which are independent of the actual choice of double coset representatives $\Delta \!\setminus\! \sym{pk}/ D$.  As such we may assume that, by Proposition \ref{P:choice}(i), whenever $x \in \Delta \!\setminus\! \sym{pk}/ D$ such that $p \mid |xDx^{-1} \cap \Delta|$, we have $x \in (\prod_{r=1}^k C_p[r]) \sym{k}^{[p]}$.  Thus, by Proposition \ref{P:choice}(ii), we have
\begin{align*}
\Res^{\sym{pk}}_{\Delta} \Lambda_k &\cong \bigoplus_{x \in \Gamma_1} \Ind_{\Delta}^{\Delta}\, (x\otimes(\F \boxtimes \F_{\delta})) \oplus \bigoplus_{x \in \Gamma_2} \Ind_{\Delta_k(C_p)}^{\Delta}\, (x\otimes(\F \boxtimes \F_{\delta})) \oplus P \\
&\cong \F^{\oplus |\Gamma_1|} \oplus (\Ind_{C_p}^{\sym{p}} \F)^{\oplus |\Gamma_2|} \oplus P
\end{align*}
for some projective $\F\Delta$-module $P$. The final isomorphism is obtained by using Proposition \ref{P:Ind F}.
\end{proof}

Keeping the notations of Theorem \ref{T:Res}, we may assume that $1 \in \Gamma_1$.  Thus $\Gamma_1 \ne \emptyset$. Using Theorem \ref{T:Res} and Lemma \ref{L:F}, we obtain the follow corollary.



\begin{cor} \label{C:Res}
Let $k \in \mathbb{Z}^+$ with $p \nmid k$.  Then
$\Res^{\sym{pk}}_{\Delta_k(\sym{p})} \Lambda_k$ has the same period as the trivial $\F \sym{p}$-module $\F$.
\end{cor}

We are now ready to prove Theorem \ref{T:main}(ii).

\begin{proof}[Proof of 
Theorem \ref{T:main}(ii)]
Let the actual period of $\Lie(pk)$ be $l_p$ and the asserted period be $n_p$. By Theorem \ref{T:main}(i), we have $l_p\leq n_p$. On the other hand, Lemma \ref{L:period}(ii,iii) and Corollary \ref{C:Res} show that $l_p \geq n_p$.  Thus $l_p = n_p$ as desired.
\end{proof}

We end the paper with the following remarks.

\begin{rems} \hfill
\begin{enumerate}
\item One may use Theorem \ref{T:Res} to show that, when $p$ is odd, $\Omega^i(\Res^{\sym{pk}}_{\Delta_k(\sym{p})} \Lambda_k) \not\cong \Omega^j(\Res^{\sym{pk}}_{\Delta_k(\sym{p})} \Lambda_k)$ if $0 < |i-j| < 2p-2$.  This implies that 
    $$
    \Omega^i(\Lie(pk)) \cong  \Omega^{i+1}(\Lambda_k) \not\cong \Omega^{j+1}(\Lambda_k) \cong \Omega^j(\Lie(pk))$$ whenever $0 < |i-j| < 2p-2$.  Together with Theorem \ref{T:main}(i) (or Theorem \ref{T: Heller}), this gives an alternative proof of Theorem \ref{T:main}(ii).
\item If the conjecture made in \cite{ELT} on the complexities of the Lie modules holds true; namely, the complexity of the Lie module $\Lie(p^mk)$ is $m$ whenever $p\nmid k$, then $\Lie(pk)$ with $p \nmid k$ are the only ones with complexity $1$, and hence our results deal with {\em all} periodic Lie modules.
\end{enumerate}
\end{rems}

\end{document}